\documentclass[11pt,oneside, a4paper]{amsart} 

\usepackage[english]{babel}
\usepackage[T1]{fontenc}
\usepackage{amsmath}
\usepackage{amssymb}
\usepackage{float}
\usepackage{amsfonts}
\usepackage{mathtools}
\usepackage[utf8]{inputenc}
\usepackage[mathcal]{eucal}
\usepackage{enumerate}
\usepackage{amsthm}
\usepackage{hyperref}
\usepackage[totalwidth=14cm,totalheight=20cm]{geometry}
\usepackage{epsfig,fancyhdr,color}
\usepackage{multicol} 
\usepackage{csquotes}

\numberwithin{equation}{section}

\newtheorem{cor}{Corollary}[section]
\newtheorem{corx}{Corollary}

\newtheorem{teo}[cor]{Theorem}
\newtheorem{thmx}[corx]{Theorem}

\newtheorem{prop}[cor]{Proposition}

\newtheorem{lemma}[cor]{Lemma}

\theoremstyle{definition}
\newtheorem{defi}[cor]{Definition}
\theoremstyle{remark}
\newtheorem{remark}[cor]{Remark}
\newtheorem*{remark*}{Remark}

\newcommand{\R}{\mathbb{R}}

\newcommand{\C}{\mathbb{C}}

\newcommand{\h}{\mathbb{H}}

\newcommand{\SL}{\mathrm{SL}}

\newcommand{\AdS}{\mathrm{AdS}}

\newcommand{\dPSL}{\mathbb{P}\mathrm{SL}(2,\mathbb{R})\times \mathbb{P}\mathrm{SL}(2,\mathbb{R})}
\newcommand{\PSL}{\mathbb{P}\mathrm{SL}}

\newcommand{\Ima}{\mathrm{Im}}

\newcommand{\Isom}{\mathrm{Isom}}

\newcommand{\trace}{\mathrm{tr}}

\newcommand{\hol}{\mathrm{hol}}
\newcommand{\Teich}{\mathcal{T}}
\newcommand{\dTeich}{\mathcal{T}(S)\times \mathcal{T}(S)}

\DeclareMathAlphabet{\mathpzc}{OT1}{pzc}{m}{it}

\title[Riemannian metrics on $\mathcal{GH}(S)$]{Riemannian metrics on the moduli space of GHMC anti-de Sitter structures}

\author{Andrea Tamburelli}

\begin{document}

\begin{abstract} We first extend the construction of the pressure metric to the deformation space of globally hyperbolic maximal Cauchy-compact anti-de Sitter structures. We show that, in contrast with the case of the Hitchin components, the pressure metric is degenerate and we characterize its degenerate locus. We then introduce a nowhere degenerate Riemannian metric adapting the work of Qiongling Li on the $\SL(3,\R)$-Hitchin component to this moduli space. We prove that the Fuchsian locus is a totally geodesic copy of Teichmüller space endowed with a multiple of the Weil-Petersson metric. 
\end{abstract}

\date{\today}
\maketitle
\setcounter{tocdepth}{1}

\tableofcontents

\section*{Introduction}
Let $S$ be a closed, connected, oriented surface of negative Euler characteristic. The aim of this short note is to introduce two Riemannian metrics on the deformation space $\mathcal{GH}(S)$ of convex co-compact anti-de Sitter structures on $S\times \R$. These are the geometric structures relevant for the study of pairs of conjugacy classes of representations $\rho_{L,R}:\pi_{1}(S) \rightarrow \PSL(2,\R)$ that are faithful and discrete.\\

In recent years, much work has been done in order to understand the geometry of $\mathcal{GH}(S)$ (\cite{Tambu_FN}, \cite{Tambu_rays}, \cite{Tambu_pinching}, \cite{Tambu_regular}, \cite{Tambu_poly}, \cite{Charles}). It turns out that many of the phenomena described in the aformentioned papers have analogous counterparts in the theory of Hitchin representations in $\mathrm{SL}(3,\R)$ (\cite{LZ}, \cite{Loftin_neck}, \cite{Loftin_rays},  \cite{Loftin_compact}, \cite{DW}, \cite{TW}, \cite{CT}). Pushing this correspondence even further, we explain in this paper how to construct two Riemannian metrics in $\mathcal{GH}(S)$ following analogous constructions known for the $\SL(3,\R)$-Hitchin component.\\

The first Riemannian metric we define is the pressure metric introduced by Brigdeman, Canary, Labourie and Sambarino (\cite{pressure_metric}, \cite{pressure_metric_survey}) for the Hitchin components, and inspired by previous work of Bridgeman (\cite{B_pressure}) on quasi-Fuchsian representations and McMullen's thermodynamic interpretation (\cite{McMullen_pressure}) of the Weil-Petersson metric on Teichmüller space. Although the construction of the pressure metric in $\mathcal{GH}(S)$ can be carried out analogously, we show that, unlike in the Hitchin components, the pressure metric is degenerate and we characterize its degenerate locus:

\begin{thmx} The pressure metric on $\mathcal{GH}(S)$ is degenerate only at the Fuchsian locus along pure bending directions.
\end{thmx}

Here, the Fuchsian locus in $\mathcal{GH}(S)$ consists of pairs of discrete and faithful representations of $\pi_{1}(S)$ that coincide up to conjugation, and pure bending directions correspond to deformations of representations away from the Fuchsian locus that are analogs of bending deformations for quasi-Fuchsian representations in $\PSL(2,\C)$ (\cite{Tambu_FN}). \\

The second Riemannian metric we define follows instead the contruction of Li on the $\SL(3,\R)$-Hitchin component (\cite{Li_metric}) and it is based on the introduction of a preferred $\rho$-equivariant scalar product in $\R^{4}$ for a given $\rho \in \mathcal{GH}(S)$. The main result is the following:

\begin{thmx} This Riemannian metric is nowhere degenerate in $\mathcal{GH}(S)$ and restricts to a multiple of the Weil-Petersson metric on the Fuchsian locus, which, moreover, is totally geodesic.
\end{thmx}

\section{Pressure metric on $\mathcal{GH}(S)$}
In this section we adapt the construction of the pressure metric on the Hitchin component (\cite{pressure_metric}, \cite{pressure_metric_survey}) to the deformation space of globally hyperbolic maximal Cauchy-compact anti-de Sitter manifolds. We will show that the pressure metric is degenerate at the Fuchsian locus along ``pure bending'' directions.

\subsection{Background on anti-de Sitter geometry} We briefly recall some notions of anti-de Sitter geometry that will be used in the sequel.\\

The $3$-dimensional anti-de Sitter space $\AdS_{3}$ is the local model of Lorentzian manifolds of constant sectional curvature $-1$ and can be defined as the set of projective classes of time-like vectors of $\R^{4}$ endowed with a bilinear form of signature $(2,2)$.

We are interested in a special class of spacetimes locally modelled on $\AdS_{3}$, introduced by Mess  (\cite{Mess}), called Globally Hyperbolic Maximal Cauchy-compact (GHMC). This terminology comes from physics and indicates that these spacetimes contain an embedded space-like surface that interesects any inextensible causal curve in exactly one point. From a modern mathematical point of view (\cite{DGK_convex_cocompact_pseudo}), we can describe these manifolds as being convex co-compact anti-de Sitter manifolds diffeomorphic to $S\times \R$, where $S$ is a closed surface of genus at least $2$. This means that, identifying the fundamental group of $S$ with a discrete subgroup $\Gamma$ of $\Isom_{0}(\AdS_{3})\cong \dPSL$ via the holonomy representation $\hol:\pi_{1}(S)\rightarrow \dPSL$, the group $\Gamma$ acts properly disontinuously and co-compactly on a convex domain in $\AdS_{3}$. \\

We denote by $\mathcal{GH}(S)$ the deformation space of globally hyperbolic maximal Cauchy-compact anti-de Sitter structures on $S\times \R$. It turns out that the holonomy of a GHMC anti-de Sitter manifold into $\dPSL$ is faithful and discrete in each factor. Moreover, we have a homeomorphism between $\mathcal{GH}(S)$ and the product $\dTeich$ of two copies of the Teichm\"uller space of $S$ (\cite{Mess}).  In particular, each simple closed curve $\gamma \in \pi_{1}(S)$ is sent by the holonomy representation $\rho=(\rho_{L}, \rho_{R})$ to a pair of hyperbolic isometries of $\h^{2}$, which preserves a space-like geodesic in the convex domain of discontinuity of $\rho$ in $\AdS_{3}$, on which $\rho(\gamma)$ acts by translation by 
\[
    \ell_{\rho}(\gamma)=\frac{1}{2}(\ell_{\rho_{L}}(\gamma)+\ell_{\rho_{R}}(\gamma)) \ .
\]
We will refer to $\ell_{\rho}(\gamma)$ as the translation length of the isometry $\rho(\gamma)$ (see \cite{Tambu_FN}). \\

We will say that the holonomy $\rho: \pi_{1}(S) \rightarrow \dPSL$ of a GHMC anti-de Sitter structure is Fuchsian if, up to conjugation, its left and right projections coincide.

\subsection{Background on thermodynamical formalism}
Let $X$ be a Riemannian manifold. A smooth flow $\phi=(\phi_{t})_{t \in \R}$ is Anosov if there is a flow-invariant splitting 
$TX=E^{s}\oplus E_{0}\oplus E^{u}$, where $E_{0}$ is the bundle parallel to the flow and, for $t\geq 0$, the differential $d\phi_{t}$ exponentially contracts $E^{s}$ and exponentially expands $E^{u}$. We say that $\phi$ is topologically transitive if it has a dense orbit.\\

Given a periodic orbit $a$ for the flow $\phi$, we denote by $\ell(a)$ its period. Let $f:X \rightarrow \R$ be a positive H\"older function. It is possible (\cite{pressure_metric}) to reparametrize the flow $\phi$ and obtain a new flow $\phi^{f}$ with the property that each closed orbit $a$ has period
\[
    \ell_{f}(a):=\int_{0}^{\ell(a)}f(\phi_{s}(x))ds \ \ \ \ \ \ x \in a \ .
\]
We define 
\begin{itemize}
 \item the topological entropy (\cite{Bowen_hyperbolic}) of $f$ as
 \[
    h(f)=\limsup_{T \to +\infty} \frac{\log(|R_{T}(f)|)}{T}
 \]
 where $R_{T}=\{ a \ \text{closed orbit of $\phi$} \ | \ \ell_{f}(a)\leq T\}$;
 \item the topological pressure (\cite{BR_ergodic}) of a H\"older function $g$ (not necessarily positive) as
 \[
    P(g)=\limsup_{T\to +\infty}\frac{1}{T}\log\left(\sum_{a \in R_{T}}e^{\ell_{g}(a)}\right) \ . 
 \]
\end{itemize}
These two notions are related by the following result:
\begin{lemma}[\cite{Sambarino_convex_reps}]\label{lm:entropy} Let $\phi$ be a topologically transitive Anosov flow on $X$ and let $f:X\rightarrow \R$ be a positive H\"older function. Then $P(-hf)=0$ if and only if $h=h(f)$.
\end{lemma}
 
Consider then the space
\[
    \mathcal{P}(X)=\{f: X \rightarrow (0,+\infty) \ | \ \text{$f$ H\"older}, \ P(f)=0\}
\]
and its quotient $\mathcal{H}(X)$ by the equivalence relation that identifies positive H\"older functions with the same periods. The analitic regularity of the pressure (\cite{PP_zeta}, \cite{thermo_formalism}) allows to define the \emph{pressure metric} on $T_{f}\mathcal{P}(X)$ as
\[
    \|g\|_{P}^{2}=-\frac{\frac{d^{2}}{dt^{2}}P(f+tg)_{|_{t=0}}}{\frac{d}{dt}P(f+tg)_{|_{t=0}}} \ .
\]

\begin{teo}[\cite{PP_zeta}, \cite{thermo_formalism}]\label{thm:pressure} Let $X$ be a Riemannian manifold endowed with a topologically transitive Anosov flow. Then the pressure metric on $\mathcal{H}(X)$ is positive definite.
\end{teo}

\noindent In particular, given a one parameter family of positive H\"older functions $f_{t}:X \rightarrow \R$, the functions $\Phi(t)=-h(f_{t})f_{t}$ describe a path in $\mathcal{P}(X)$ by Lemma \ref{lm:entropy} and $\| \dot{\Phi}\|_{P}=0$ if and only if $\dot{\Phi}$ has vanishing periods, hence if and only if $\frac{d}{dt}_{|_{t=0}}h(f_{t})\ell_{f_{t}}(a)=0$ for every closed orbit $a$ for $\phi$.

\subsection{Pressure metric on $\mathcal{GH}(S)$} We apply the above theory to the unit tangent bundle $X=T^{1}S$ of a hyperbolic surface $(S, \rho_{0})$ endowed with its geodesic flow $\phi=\phi^{\rho_{0}}$. Here, $\rho_{0}:\pi_{1}(S) \rightarrow \PSL(2,\R)$ is a fixed Fuchsian representation that defines a marked hyperbolic metric on $S$. We also fix an identification of the universal cover $\tilde{S}$ with $\h^{2}$, and, consequently, of the Gromov boundary $\partial_{\infty}\pi_{1}(S)$ of the fundamental group with $S^{1}$. The following facts are well-known from hyperbolic geometry:

\begin{prop}[\cite{pressure_metric_survey}] If $\rho, \eta: \pi_{1}(S) \rightarrow \PSL(2,\R)$ are two Fuchsian representations, then there is a unique $(\rho,\eta)$-equivariant H\"older homeomorphism $\xi_{\rho,\eta}:\partial_{\infty}\h^{2}\rightarrow \partial_{\infty}\h^{2}$ that varies analytically in $\eta$.
\end{prop}

\begin{prop}[\cite{pressure_metric_survey}]\label{prop:rep_hyp} For every Fuchsian representation $\eta$, there is a positive H\"older function $f_{\eta}:X\rightarrow \R$ with period $\ell_{f_{\eta}}(\gamma)$ coinciding with the hyperbolic length $\ell_{\eta}(\gamma)$ of the closed geodesic $\gamma$ for the hyperbolic metric induced by $\eta$. Moreover, $f_{\eta}$ varies analytically in $\eta$.
\end{prop}

\begin{cor} For every $\rho=(\eta_{L},\eta_{R}) \in \mathcal{GH}(S)$, there exists a positive H\"older function $f_{\rho}:X\rightarrow \R$ such that $\ell_{f_{\rho}}(\gamma)=\ell_{\rho}(\gamma)$ for every simple closed curve $\gamma \in \pi_{1}(S)$.
\end{cor}
\begin{proof} Recall that $\ell_{\rho}(\gamma)=\frac{1}{2}(\ell_{\eta_{L}}(\gamma)+\ell_{\eta_{R}}(\gamma))$, thus it is sufficient to choose $f_{\rho}=\frac{1}{2}(f_{\eta_{L}}+f_{\eta_{R}})$. Moreover, $f_{\rho}$ varies analytically in $\rho$ by Proposition \ref{prop:rep_hyp}.
\end{proof}

\noindent We can then introduce the termodynamic mapping:
\begin{align*}
    \Phi:\mathcal{GH}(S)&\longrightarrow \mathcal{P}(X)\\
        \rho &\mapsto -h(f_{\rho})f_{\rho} \ .
\end{align*}
By pulling-back the pressure metric via $\Phi$, we obtain a semi-definite metric on $\mathcal{GH}(S)$, which we still call pressure metric.

\begin{prop}\label{prop:pressure_Fuchsian} The restriction of the pressure metric to the Fuchsian locus in $\mathcal{GH}(S)$ is a constant multiple of the Weil-Petersson metric.
\end{prop}
\begin{proof} Let $\rho_{t}=(\eta_{t}, \eta_{t})$ be a path on the Fuchsian locus. Then
\[
    \Phi(\rho_{t})=-h(f_{\rho_{t}})f_{\rho_{t}}=-h(f_{\rho_{t}})f_{\eta_{t}}=-f_{\eta_{t}} \ ,
\]
where in the last step we used the fact that the entropy of a Fuchsian representation is $1$ (\cite{Tambu_rays}, \cite{GM}). Therefore, $d\Phi(\dot\rho_{0})=-\dot{f}_{\eta_{0}}$ and the result follows from (\cite{McMullen_pressure}). 
\end{proof}

\begin{defi} Let $\rho=(\eta, \eta)$ be a Fuchsian representation. We say that a tangent vector $w\in T_{\rho}\mathcal{GH}(S)$ is a pure bending direction if $w=(v, -v)$ for some $v\in T_{\eta}\Teich(S)$.
\end{defi}

\begin{lemma} The pressure metric on $\mathcal{GH}(S)$ is degenerate on the Fuchsian locus along pure bending directions.
\end{lemma}
\begin{proof} Let $\rho_{t}=(\eta_{L,t}, \eta_{R,t})$ be a path in $\mathcal{GH}(S)$ such that $\rho_{0}$ is Fuchsian and $\dot{\rho}_{0}=\frac{d}{dt}_{|_{t=0}}\rho_{t}=(v,-v)$ for some $v\in T_{\eta_{0}}\Teich(S)$. By definition of the pressure metric and Theorem \ref{thm:pressure}, we have $\| d\Phi(\dot{\rho}_{0})\|=0$ if and only if $\frac{d}{dt}_{|_{t=0}}h(\rho_{t})\ell_{\rho_{t}}(\gamma)=0$ for every closed geodesic $\gamma$ on $S$. By the product rule and the fact that the entropy is maximal and equal to $1$ at the Fuchsian locus (\cite{Tambu_rays}, \cite{CTT}), we get
\begin{align*}
 \frac{d}{dt}_{|_{t=0}}h(\rho_{t})\ell_{\rho_{t}}(\gamma)&=\frac{d}{dt}_{|_{t=0}}\ell_{\rho_{t}}(\gamma) \\
 &=\frac{1}{2}\left(\frac{d}{dt}_{|_{t=0}}\ell_{\eta_{L,t}}(\gamma)+\frac{d}{dt}_{|_{t=0}}\ell_{\eta_{R,t}}(\gamma)\right)\\
 &=\frac{1}{2}(d\ell_{\eta_{0}}(v)+d\ell_{\eta_{0}}(-v))=0 \ .
\end{align*}
\end{proof}

\begin{remark}\label{rmk:Bridgeman}As remarked in \cite{B_pressure}, we note that, along a general path $\rho_{t} \in \mathcal{GH}(S)$, the condition $\frac{d}{dt}_{|_{t=0}}h(\rho_{t})\ell_{\rho_{t}}(\gamma)=0$ for every closed geodesic $\gamma$ is equivalent to the existence of a constant $k \in \R$ such that
\[
 \frac{d}{dt}_{|_{t=0}}\ell_{\rho_{t}}(\gamma)=k\ell_{\rho_{0}}(\gamma) \ .
\]
In fact, $k=-\frac{1}{h(\rho_{0})}\frac{d}{dt}_{|_{t=0}}h(\rho_{t})$.
\end{remark}
 
\begin{lemma}\label{lm:technical} Let $v\in T_{\rho}\mathcal{GH}(S)$ be a non-zero vector.
If there exists $k \in \R$ such that
\begin{equation}\label{eq:length}
 \frac{d}{dt}_{|_{t=0}}\ell_{\rho_{t}}(\gamma)=k\ell_{\rho_{0}}(\gamma) \ 
\end{equation}
for every closed geodesic $\gamma$, then $k=0$ or $\rho$ is Fuchsian.
\end{lemma}
\begin{proof}We show that if $\rho$ is not Fuchsian, then $k$ is necessarily $0$. The proof follows the line of \cite[Lemma 7.4]{B_pressure}. Let $v=(v_{1},v_{2})$ and $\rho_{t}=(\rho_{1,t}, \rho_{2,t})$. Choose simple closed curves $\alpha$ and $\beta$ in $S$. Up to conjugation we can assume that 
\[
    A_{i}(t)=\rho_{i,t}(\alpha)=\begin{pmatrix} \lambda_{i}(t) & 0 \\ 0 & \lambda_{i}(t)^{-1}
    \end{pmatrix} \ ,
\]
where we denoted by $\lambda_{i}(t)$ the largest eigenvalue of the hyperbolic isometry $\rho_{i,t}(\alpha)$. 
Let 
\[
    B_{i}(t)=\rho_{i,t}(\beta)=\begin{pmatrix} a_{i}(t) & b_{i}(t) \\ c_{i}(t) & d_{i}(t)
    \end{pmatrix}
\]
such that $\det(B_{i}(t))=1$ and $\trace(B_{i}(t))>2$. Notice that $b_{i}(t)c_{i}(t)\neq 0$ because $B_{i}(t)$ is hyperbolic and $A_{i}(t)$ and $B_{i}(t)$ have different axis. For every $n\geq 0$, we consider the matrices
\[
    C_{i,n}(t)=A_{i}^{n}(t)B_{i}(t)=\rho_{i,t}(\gamma_{n})=\begin{pmatrix} \lambda_{i}(t)^{n}a_{i}(t) & \lambda_{i}(t)^{n}b_{i}(t) \\ \lambda_{i}(t)^{-n}c_{i}(t) & \lambda_{i}(t)^{-n}d_{i}(t)
    \end{pmatrix}
\]
associated to some closed curves $\gamma_{n}$ on $S$. The eigenvalues $\mu_{i,n}$ of $C_{i,n}(t)$ satisfy
\[
 \log(\mu_{i,n}(t))=n\log(\lambda_{i}(t))+\log(a_{i}(t))+\lambda_{i}(t)^{-2n}\left(\frac{a_{i}(t)d_{i}(t)-1}{a_{i}(t)^{2}}\right)+O(\lambda_{i}(t)^{-4n})
\]
as $n \to +\infty$.
Applying Equation (\ref{eq:length}) to the curves $\gamma_{n}$, we obtain
\begin{align*}
    0&=\frac{d}{dt}_{|_{t=0}}\ell_{\rho_{t}}(\gamma_{n})-k\ell_{\rho_{0}}(\gamma_{n})\\ 
     &= n\log(\lambda_{1}\lambda_{2})'-kn\log(\lambda_{1}\lambda_{2})\\
     &+\log(a_{1}a_{2})'-k\log(a_{1}a_{2})\\
     &-2n\left[\lambda_{1}^{-2n-1}\lambda_{1}'\left(\frac{a_{1}d_{1}-1}{a_{1}^{2}}\right)+\lambda_{2}^{-2n-1}\lambda_{2}'\left(\frac{a_{2}d_{2}-1}{a_{2}^{2}}\right)\right]\\
     &+\lambda_{1}^{-2n}\left[\left(\frac{a_{1}d_{1}-1}{a_{1}^{2}}\right)'-k\left(\frac{a_{1}d_{1}-1}{a_{1}^{2}}\right)\right]\\
     &+\lambda_{2}^{-2n}\left[\left(\frac{a_{2}d_{2}-1}{a_{2}^{2}}\right)'-k\left(\frac{a_{2}d_{2}-1}{a_{2}^{2}}\right)\right]+o(\lambda_{i}^{-2n})
\end{align*}
 where all derivatives and all functions are intended to be taken and evaluated at $t=0$.
The term $n\log(\lambda_{1}\lambda_{2})'-kn\log(\lambda_{1}\lambda_{2})$ vanishes by assumption because
\[
    \ell_{\rho_{t}}(\alpha)=\frac{1}{2}(\ell_{\rho_{1,t}}(\alpha)+\ell_{\rho_{2,t}}(\alpha))=\log(\lambda_{1}(t)\lambda_{2}(t))
\]
and Equation (\ref{eq:length}) holds for the curves $\alpha$.
Taking the limit of the above expression as $n\to +\infty$, we deduce that $\log(a_{1}a_{2})'-k\log(a_{1}a_{2})=0$. Because $\rho$ is not Fuchsian, we can assume to have chosen $\alpha$ and $\beta$ so that $\lambda_{1}(0)>\lambda_{2}(0)$. Then if we multiply the equation above by $\frac{\lambda_{2}^{2n}}{n}$ and take the limit as $n \to +\infty$, we deduce that $\lambda_{2}'=0$. Similarly, multiplying by $\frac{\lambda_{1}^{2n}}{n}$, we find that $\lambda_{1}'=0$. 
Therefore,
\[
    \frac{d}{dt}_{|_{t=0}}\ell_{\rho_{t}}(\alpha)=\frac{d}{dt}_{|_{t=0}}\log(\lambda_{1}(t)\lambda_{2}(t))=0 \ ,
\]
hence $k=0$.
\end{proof}

\begin{teo} Let $v=(v_{L}, v_{R})\in T_{\rho}\mathcal{GH}(S)$ be a non-zero tangent vector such that $\| d\Phi(v)\|=0$. Then $\rho$ is Fuchsian and $v$ is a pure bending direction.
\end{teo}
\begin{proof} Let $\rho_{t}$ be a path in $\mathcal{GH}(S)$ such that $\rho_{0}=\rho$ and $\rho_{t}$ is tangent to $v$. If $\rho=(\eta, \eta)$ is Fuchsian, then, combining Remark \ref{rmk:Bridgeman} with the fact that the entropy is maximal and equal to $1$ at the Fuchsian locus, we get
\[
  0=\frac{d}{dt}_{|_{t=0}}h(\rho_{t})\ell_{\rho_{t}}(\gamma)=\frac{d}{dt}_{|_{t=0}}\ell_{\rho_{t}}(\gamma)
\]
for every simple closed geodesic $\gamma$ in $S$. Therefore, 
\[
 0=\frac{d}{dt}_{|_{t=0}}\ell_{\rho_{t}}(\gamma)=\frac{1}{2}(d\ell_{\eta}(\gamma)(v_{L})+d\ell_{\eta}(\gamma)(v_{R}))
\]
from which we deduce that $v_{R}=-v_{L}$, because $\{d\ell_{\eta}(\gamma)\}_{\gamma}$ generates $T^{*}_{\eta}\Teich(S)$. Hence, $v$ is a pure-bending direction.\\
We are thus left to show that $\rho$ is necessarily Fuchsian. Suppose it is not and denote with $\rho_{L}\neq \rho_{R}$ the projections of $\rho$. By the previous lemma
\[
 \ell_{\rho}'(\gamma)=\frac{d}{dt}_{|_{t=0}}\ell_{\rho_{t}}(\gamma)=0
\]
for every simple closed geodesic $\gamma$ in $S$. Moreover, we have shown in the proof of Lemma \ref{lm:technical} that if $\ell_{\rho_{L}}(\gamma)\neq \ell_{\rho_{R}}(\gamma)$ then $\ell_{\rho_{L}}'(\gamma)=\ell_{\rho_{R}}'(\gamma)=0$. Otherwise, $\ell_{\rho_{L}}'(\gamma)=-\ell_{\rho_{R}}'(\gamma)$. Exploiting the isomorphism $\mathrm{SL}(2,\R)\times \SL(2,\R) \cong \mathrm{SO}_{0}(2,2)$, we find that the matrix $\rho_{t}(\gamma)$ is conjugated to
\[
 \exp\left(\frac{1}{2}\mathrm{diag}(\ell_{\rho_{L,t}}+\ell_{\rho_{R,t}},
 \ell_{\rho_{L,t}}-\ell_{\rho_{R,t}},
 \ell_{\rho_{R,t}}-\ell_{\rho_{L,t}},
 -\ell_{\rho_{L,t}}-\ell_{\rho_{R,t}}) \right)\ ,
\]
thus
\[
   d\trace(\rho(\gamma))(v)= \frac{d}{dt}_{|_{t=0}}\trace(\rho_{t}(\gamma))=0
\]
for every simple closed geodesic $\gamma$. Because $\rho$ is generic in the sense of \cite[Proposition 10.3]{pressure_metric}, the differentials of traces $\{d\trace(\rho(\gamma))\}_{\gamma}$ generate $T_{\rho}^{*}\mathcal{GH}(S)$ and we must have $v=0$. 
\end{proof}

\section{A non-degenerate Riemannian metric on $\mathcal{GH}(S)$}
In this section we define a non-degenerate Riemannian metric on $\mathcal{GH}(S)$ following Li's construction (\cite{Li_metric}) for the $\SL(3.\R)$-Hitchin component.

\subsection{Preliminaries} In this section we identify $\mathcal{GH}(S)$ with a connected component of the space of representations $\mathrm{Hom}(\pi_{1}(S), \mathrm{SO}_{0}(2,2))/\mathrm{SO}_{0}(2,2)$ via the holonomy map. Recall that by Mess' parametrization (\cite{Mess}), this component is smooth and diffeomorphic to $\Teich(S)\times \Teich(S)$. This allows to identify the tangent space $T_{\rho}\mathcal{GH}(S)$ at $\rho \in \mathcal{GH}(S)$ with the cohomology group $H^{1}(S, \mathfrak{so}_{0}(2,2)_{\mathrm{Ad}\rho})$, where $\mathfrak{so}_{0}(2,2)_{\mathrm{Ad}\rho}$ denotes the flat $\mathfrak{so}_{0}(2,2)$ bundle over $S$ with holonomy $\mathrm{Ad}\rho$. Explicitly, 
\[
    \mathfrak{so}_{0}(2,2)_{\mathrm{Ad}\rho} = (\tilde{S} \times \mathfrak{so}_{0}(2,2))/\sim
\]
where $(\tilde{x},v)\sim (\gamma \tilde{x}, \mathrm{Ad}\rho(\gamma)(v))$ for any $\gamma \in \pi_{1}(S)$, $x \in \tilde{S}$ and $v \in \mathfrak{so}_{0}(2,2)$. \\

In order to define a Riemannian metric on $\mathcal{GH}(S)$ is thus sufficient to introduce a non-degenerate scalar product on $H^{1}(S, \mathfrak{so}_{0}(2,2)_{\mathrm{Ad}\rho})$. Let us assume for the moment that we have chosen an inner product $\iota$ on the bundle $\mathfrak{so}_{0}(2,2)_{\mathrm{Ad}\rho}$ and a Riemannian metric $h$ on $S$. A Riemannian metric in cohomology follows then by standard Hodge theory that we recall briefly here.
The Riemannian metric $h$ and the orientation on $S$ induce a scalar product $\langle \cdot, \cdot \rangle$ on the space $\mathcal{A}^{p}(S)$ of $p$-forms on $S$, which allows to define a Hodge star operator  
\[
    *: \mathcal{A}^{p}(S) \rightarrow \mathcal{A}^{2-p}(S)
\]
by setting 
\[
    \alpha \wedge (*\beta)=\langle \alpha, \beta\rangle_{h} dA_{h} \ .
\]
This data gives a bi-linear pairing $\tilde{g}$ in the space of $\mathfrak{so}_{0}(2,2)_{\mathrm{Ad}\rho}$-valued $1$-forms as follows:
\[
    \tilde{g}(\sigma\otimes \phi, \sigma'\otimes \phi')=\int_{S}\iota(\phi, \phi')\sigma \wedge(*\sigma') \ ,
\]
where $\sigma, \sigma' \in \mathcal{A}^{1}(S)$ and $\phi, \phi'$ are sections of $\mathfrak{so}_{0}(2,2)_{\mathrm{Ad}\rho}$. \\
Given $\rho \in \mathcal{GH}(S)$, we denote by $\rho^{*}$ the contragradient representation (still into $\mathrm{SO}_{0}(2,2)$) defined by $(\rho^{*}(\gamma)L)(v)=L(\rho^{-1}(\gamma)v)$ for every $v \in \R^{4}$ and $L \in \R^{4*}=\mathrm{Hom}(\R^{4}, \R)$. The flat bundle $\mathfrak{so}_{0}(2,2)_{\mathrm{Ad}\rho^{*}}$ is dual to $\mathfrak{so}_{0}(2,2)_{\mathrm{Ad}\rho}$ and the inner product $\iota$ induces an isomorphism (\cite{R_book})
\[
    \#:\mathfrak{so}_{0}(2,2)_{\mathrm{Ad}\rho} \rightarrow \mathfrak{so}_{0}(2,2)_{\mathrm{Ad}\rho^{*}}
\]
defined by setting
\[
    (\#A)(B)=\iota(A,B)
\]
for $A,B \in \mathfrak{so}_{0}(2,2)$. This extends naturally to an isomorphism
\[
    \#: \mathcal{A}^{p}(S, \mathfrak{so}_{0}(2,2)_{\mathrm{Ad}\rho}) \rightarrow \mathcal{A}^{p}(S, \mathfrak{so}_{0}(2,2)_{\mathrm{Ad}\rho^{*}}) \ .
\]
Consequently, we can introduce a coboundary map 
\[
 \delta:\mathcal{A}^{p}(S, \mathfrak{so}_{0}(2,2)_{\mathrm{Ad}\rho}) \rightarrow \mathcal{A}^{p-1}(S, \mathfrak{so}_{0}(2,2)_{\mathrm{Ad}\rho})
\] 
by setting $\delta=-(\#)^{-1}*^{-1}d*\#$, and then a Laplacian operator 
\[
    \Delta:\mathcal{A}^{p}(S, \mathfrak{so}_{0}(2,2)_{\mathrm{Ad}\rho}) \rightarrow \mathcal{A}^{p}(S, \mathfrak{so}_{0}(2,2)_{\mathrm{Ad}\rho})
\]
given by $\Delta=d\delta+\delta d$.
A $1$-form $\xi$ is said to be harmonic if $\Delta \xi=0$, or, equivalently, if $d\xi=\delta \xi=0$. We have an orthogonal decomposition
\[
    \mathcal{A}^{1}(S, \mathfrak{so}_{0}(2,2)_{\mathrm{Ad}\rho})=\mathrm{Ker}(\Delta)\oplus \Ima(d)\oplus \Ima(\delta) 
\]
and by the non-abelian Hodge theory (\cite{R_book}) every cohomology class contains a unique harmonic representative. Therefore, the bi-linear pairing $\tilde{g}$ induces a scalar product in cohomology by setting
\begin{align*}
    g:H^{1}(S, \mathfrak{so}_{0}(2,2)_{\mathrm{Ad}\rho})&\times H^{1}(S, \mathfrak{so}_{0}(2,2)_{\mathrm{Ad}\rho})\rightarrow \R \\ 
    ([\alpha], [\beta]) &\mapsto \tilde{g}(\alpha_{harm}, \beta_{harm}) \ ,
\end{align*}
where $\alpha_{harm}$ and $\beta_{harm}$ are the harmonic representatives of $\alpha$ and $\beta$.

\subsection{Definition of the metric} As explained before, in order to define a Riemannian metric on $\mathcal{GH}(S)$ it is sufficient to define a Riemannian metric $h$ on $S$ and a scalar product $\iota$ on $\mathfrak{so}_{0}(2,2)_{\mathrm{Ad}\rho}$.\\

Let us begin with the metric $h$ on $S$. Given $\rho \in \mathcal{GH}(S)$, we denote by $M_{\rho}$ the unique GHMC anti-de Sitter manifold with holonomy $\rho$, up to isotopy. It is well-known that $M_{\rho}$ contains a unique embedded maximal (i.e. with vanishing mean curvature) surface $\Sigma_{\rho}$ (\cite{BBZ}). A natural choice for $h$ is thus the induced metric on $\Sigma_{\rho}$.\\

As for the scalar product $\iota$, we first introduce a scalar product in $\R^{4}$ that is closely related to the maximal surface and its induced metric. Lifting the surface $\Sigma_{\rho}$ to the universal cover, we can find a $\rho$-equivariant maximal embedding $\tilde{\sigma}:\tilde{S} \rightarrow \widehat{\AdS}_{3}\subset \R^{4}$, where $\widehat{\AdS_{3}}$ denotes the double cover of $\AdS_{3}$ consisting of unit time-like vectors in $\R^{4}$  endowed with a bi-linear form of signature $(2,2)$. For any $\tilde{x} \in \tilde{S}$, we thus have a frame of $\R^{4}$ formed by the unit tangent vectors $u_{1}(\tilde{x})$ and $u_{2}(\tilde{x})$ to the surface at $\tilde{\sigma}(\tilde{x})$, the time-like unit normal vector $N(\tilde{x})$ at $\tilde{\sigma}(\tilde{x})$ and the position vector $\tilde{\sigma}(x)$. We can define a scalar product $\iota_{\tilde{x}}$ on $\R^{4}$ depending on the point $\tilde{x} \in \tilde{S}$ by declaring the frame $\{u_{1}(\tilde{x}), u_{2}(\tilde{x}), \tilde{\sigma}(x), N(\tilde{x})\}$ to be orthonormal for $\iota_{\tilde{x}}$. Because $\mathfrak{so}_{0}(2,2) \subset \mathfrak{gl}(4,\R)\cong \R^{4}\times \R^{4*}$, the inner product $\iota_{\tilde{x}}$ induces an inner product on $\mathfrak{so}_{0}(2,2)$ and, consequently, on the trivial bundle $\tilde{S}\times \mathfrak{so}_{0}(2,2)$ over $\tilde{S}$. This descends to a metric $\iota$ on $\mathfrak{so}_{0}(2,2)_{\mathrm{Ad}\rho}$ by setting
\[
    \iota_{p}(\phi,\phi'):=\iota_{\tilde{x}}(\tilde{\phi}_{\tilde{x}}, \tilde{\phi}'_{\tilde{x}}) \ \ \ \text{for some $\tilde{x}\in \pi^{-1}(p)$}
\]
where $p \in S$, $\pi:\tilde{S} \rightarrow S$ is the natural projection and $\tilde{\phi}_{\tilde{x}}, \tilde{\phi}'_{\tilde{x}}$ are lifts of $\phi,\phi'$ to the trivial bundle $\tilde{S}\times \mathfrak{so}_{0}(2,2)$ evaluated at $\tilde{x}$. Because $\iota_{\tilde{x}}$ is $\rho$-equivariant, it is easy to check (see \cite{Li_metric}) that $\iota_{p}$ does not depend on the choice of $\tilde{x}\in \pi^{-1}(p)$ and thus $\iota$ is a well-defined metric on the flat bundle $\mathfrak{so}_{0}(2,2)_{\mathrm{Ad}\rho}$. \\

The following lemma is useful for computations with this metric:
\begin{lemma}[\cite{Li_metric}]\label{lm:computation} Assume that we have a matrix representation $H$ of the inner product $\iota_{\tilde{x}}$ at a point $\tilde{x} \in \pi^{-1}(p)$ with respect to the canonical basis of $\R^{4}$. Then
\[
    \iota_{p}(A,B)=\trace(A^{t}H^{-1}BH) \ \ \ \text{for} \ A,B \in \mathfrak{so}_{0}(2,2) \ . 
\]
\end{lemma}

\subsection{Restriction to the Fuchsian locus} In order to compute the restriction of the metric $g$ to the Fuchsian locus, we need to understand the induced metric on the equivariant maximal surface and find a matrix representation of the inner product $\iota$. \\

If $\rho \in \mathcal{GH}(S)$ is Fuchsian, the representation preserves a totally geodesic space-like plane in $\AdS_{3}$. Realizing explicitly (the double cover of) anti-de Sitter space as
\[
    \widehat{AdS}_{3}=\{x \in \R^{4} \ | \ x_{1}^{2}+x_{2}^{2}-x_{3}^{2}-x_{4}^{2}=-1\} \ ,
\]
we can assume, up to post-composition by an isometry, that $\rho$ preserves the hyperboloid
\[
    \mathcal{H}=\{x \in \R^{4} \ | \ x_{1}^{2}+x_{2}^{2}-x_{3}^{2}=-1 \ x_{4}=0\} \ ,
\]
which is isometric to the hyperbolic plane $\h^{2}=\{ z \in \C \ | \ \Ima(z)>0\}$: an explicit isometry (\cite{Li_metric}) being 
\begin{align}\label{eq:isometry}
        f: \h^{2} &\rightarrow \mathcal{H}\subset \R^{4} \\
        (x,y) &\mapsto \left(\frac{x}{y}, \frac{x^2+y^2-1}{2y}, \frac{x^2+y^2+1}{2y}, 0\right) \ .
\end{align}
The respresentation $\rho: \pi_{1}(S) \rightarrow \mathrm{SO}_{0}(2,2)$ factors then through the standard copy of $\mathrm{SO}_{0}(2,1)$ inside $\mathrm{SO}_{0}(2,2)$, which is isomorphic to $\PSL(2,\R)$ via the map (\cite{Kim})
\begin{align}\label{eq:Lie_group}
 \Phi:\PSL(2,\R) &\rightarrow \mathrm{SO}_{0}(2,1)<\mathrm{SO}_{0}(2,2)  \\
    \begin{pmatrix}
     a & b \\
     c & d
    \end{pmatrix} &\mapsto
    \begin{pmatrix}
     ad+bc & ac-bd & ac+bd & 0 \\
    ad-cd & \frac{a^2-b^2-c^2+d^2}{2} & \frac{a^2+b^2-c^2-d^2}{2} & 0 \\
     ab+cd & \frac{a^2-b^2+c^2-d^2}{2} & \frac{a^2+b^2+c^2+d^2}{2} & 0 \\
     0 & 0 & 0 & 1 
    \end{pmatrix} \ .
\end{align}
The map $\Phi$ induces a Lie algebra homomorphism, still denoted by $\Phi$, given by
\begin{align}\label{eq:Lie_alg}
    \Phi: \mathfrak{sl}(2,\R) &\rightarrow \mathfrak{so}_{0}(2,2) \\
    \begin{pmatrix}
     a & b \\
     c & -a 
    \end{pmatrix} &\mapsto
    \begin{pmatrix}
     0 & c-b & c+b & 0 \\
     b-c & 0 & 2a & 0 \\
     b+c & 2a & 0 & 0 \\
     0 & 0 & 0 & 0 
    \end{pmatrix}
\end{align}
It follows that if $\rho(\pi_{1}(S))=\Gamma<SO_{0}(2,2)$, then the maximal surface $\Sigma_{\rho}$ is realized by $\mathcal{H}/\Gamma$ and is isometric to the hyperbolic surface $\h^{2}/\Phi^{-1}(\Gamma)$.\\

Let us now turn our attention to the scalar product $\iota$ on $\mathfrak{so}_{0}(2,2)_{\mathrm{Ad}\rho}$. Recall that $\iota$ is determined by a family of inner products $\iota_{\tilde{x}}$ on $\R^{4}$ depending on $\tilde{x} \in \tilde{S}$, which is obtained by declaring the frame $\{u_{1}(\tilde{x}), u_{2}(\tilde{x}), \tilde{\sigma}(\tilde{x}), N(\tilde{x})\}$ orthonormal. If we   identify the universal cover of $\tilde{S}$ with $\h^{2}$, the map $f$ gives an explicit $\rho$-equivariant maximal embedding of $\tilde{S}$ into $\widehat{\AdS}_{3}$. Therefore, the coordinates of the vectors tangent and normal to the embedding with respect to the canonical basis of $\R^{4}$ can be explicitly computed and the following matrix representation $H$ of $\iota_{z}$ can be obtained for any $z \in \h^{2}$ (\cite[Corollary 6.5]{Li_metric}):
\[
    H=\begin{pmatrix}
       \frac{2x^2}{y^2}+1 & \frac{x(x^2+y^2-1)}{y^2} & -\frac{x(x^2+y^2+1)}{y^2} & 0 \\
       \frac{x(x^2+y^2-1)}{y^2} & \frac{(x^2+y^2-1)^2}{2y^2}+1 &
       -\frac{(x^2+y^2-1)(x^2+y^2+1)}{2y^2} & 0 \\
       - \frac{x(x^2+y^2+1)}{y^2} & -\frac{(x^2+y^2-1)(x^2+y^2+1)}{2y^2} & \frac{(x^2+y^2+1)^2}{2y^2}-1 & 0 \\
       0 & 0 & 0 & 1 
      \end{pmatrix}
\]
with
\[
    H^{-1}=\begin{pmatrix}
       \frac{2x^2}{y^2}+1 & \frac{x(x^2+y^2-1)}{y^2} & \frac{x(x^2+y^2+1)}{y^2} & 0 \\
       \frac{x(x^2+y^2-1)}{y^2} & \frac{(x^2+y^2-1)^2}{2y^2}+1 &
       \frac{(x^2+y^2-1)(x^2+y^2+1)}{2y^2} & 0 \\
       \frac{x(x^2+y^2+1)}{y^2} & \frac{(x^2+y^2-1)(x^2+y^2+1)}{2y^2} & \frac{(x^2+y^2+1)^2}{2y^2}-1 & 0 \\
       0 & 0 & 0 & 1 
      \end{pmatrix} \ .
\]
Together with Lemma \ref{lm:computation} we obtain the following:
\begin{cor}[\cite{Li_metric}]\label{lm;computation_Fuchsian} For any $z \in \h^{2}$, after extending the definition of Lemma \ref{lm:computation} to $A,B \in \mathfrak{so}(4,\C)$ by $\iota_{z}(A,B)=\trace(A^{t}H^{-1}\overline{B}H)$, we have
 \[
    \iota_{z}\left(\Phi\begin{psmallmatrix} -z & z^2 \\ -1 & z \end{psmallmatrix}, \Phi\begin{psmallmatrix} -z & z^2 \\ -1 & z \end{psmallmatrix}\right)=16y^2 \ .
 \] 
\end{cor}

The last ingredient we need in order to describe the restriction of $g$ to the Fuchsian locus is an explicit realization of the tangent space to the Fuchsian locus inside $T\mathcal{GH}(S)$. 

\begin{lemma}\label{lm:tangent} Let $\rho \in \mathcal{GH}(S)$ be a Fuchsian representation. 
\begin{enumerate}[(i)]
 \item The tangent space at $\rho$ to the Fuchsian locus is spanned by the cohomology class of $\phi(z)dz\otimes \Phi\begin{psmallmatrix} -z\ & z^2 \\ -1 & z \end{psmallmatrix}$,
 where $\phi(z)dz^2$ is a holomorphic quadratic differential on $\Sigma_{\rho}$.
 \item the $\mathfrak{so}_{0}(2,2)_{\mathrm{Ad}\rho}$-valued $1$-forms $\phi(z)dz\otimes \Phi\begin{psmallmatrix} -z\ & z^2 \\ -1 & z \end{psmallmatrix}$ are harmonic representatives in their own cohomology class.
\end{enumerate}
\end{lemma}
\begin{proof} 
(i)\ Let $\rho'=\Phi^{-1}(\rho)$ be the corresponding Fuchsian representation in $\PSL(2,\R)$. The claim follows from the fact (\cite{Goldman_symplectic}) that the tangent space to Teichm\"uller space is generated by the $\mathfrak{sl}(2,\R)_{\mathrm{Ad}\rho'}$-valued $1$-forms $\phi(z)dz\otimes \begin{psmallmatrix} -z\ & z^2 \\ -1 & z \end{psmallmatrix}$ and thus the tangent space to the Fuchsian locus is generated by the inclusion of $H^{1}(S, \mathfrak{sl}(2,\R)_{\mathrm{Ad}\rho'})$ inside $H^{1}(S, \mathfrak{so}_{0}(2,2)_{\mathrm{Ad}\rho})$ induced by the map $\Phi$. \\
(ii) \ We need to show that $\phi(z)dz\otimes \Phi\begin{psmallmatrix} -z\ & z^2 \\ -1 & z \end{psmallmatrix}$ is $d$-closed and $\delta$-closed. The first fact has been proved in   \cite[Lemma 6.6]{Li_metric}. As for $\delta$-closedness, we will follow the lines of the aformentioned lemma. From the definition of $\delta$, it is enough to show that $d*(\#)\left(\phi(z)dz\otimes \Phi\begin{psmallmatrix} -z\ & z^2 \\ -1 & z \end{psmallmatrix}\right)=0$. By linearity
\begin{align*}
  \#\left(\phi(z)dz\otimes\Phi\begin{psmallmatrix} -z & z^2 \\ -1 & z \end{psmallmatrix}\right) 
  &=z^2\phi(z)dz\otimes \#\left(\Phi\begin{psmallmatrix} 0 & 1 \\ 0 & 0 \end{psmallmatrix}\right)-\phi(z)dz\otimes \#\left(\Phi\begin{psmallmatrix} 0 & 0 \\ 1 & 0 \end{psmallmatrix}\right) \\ &-2z\phi(z)dz\otimes \#\left(\Phi\begin{psmallmatrix} 1/2 & 0 \\ 0 & -1/2 \end{psmallmatrix}\right) \ .
\end{align*}
We then want to calculate $\#\left(\Phi\begin{psmallmatrix} 0 & 1 \\ 0 & 0 \end{psmallmatrix}\right)$,  $\#\left(\Phi\begin{psmallmatrix} 0 & 0 \\ 1 & 0 \end{psmallmatrix}\right)$ and $\#\left(\Phi\begin{psmallmatrix} 1/2 & 0 \\ 0 & -1/2 \end{psmallmatrix}\right)$. We choose a basis for $\mathfrak{so}_{0}(2,2)$ given by
\[
    E_{1}=\Phi\begin{psmallmatrix} 0 & 1 \\ 0 & 0 \end{psmallmatrix}=\begin{psmallmatrix} 0 & -1 & 1 & 0\\
        1 & 0 & 0 & 0 \\ 1 & 0 & 0 & 0 \\ 0 & 0 & 0 & 0 \end{psmallmatrix} \ \ \ 
    E_{2}=\Phi\begin{psmallmatrix} 1/2 & 0 \\ 0 & -1/2 \end{psmallmatrix}=\begin{psmallmatrix} 0 & 0 & 0 & 0\\
        0 & 0 & 1 & 0 \\ 0 & 1 & 0 & 0 \\ 0 & 0 & 0 & 0  \end{psmallmatrix}
\]
\[
    E_{3}=\Phi\begin{psmallmatrix} 0 & 0 \\ 1 & 0 \end{psmallmatrix}=\begin{psmallmatrix} 0 & 1 & 1 & 0 \\
        -1 & 0 & 0 & 0 \\ 1 & 0 & 0 & 0 \\ 0 & 0 & 0 & 0 \end{psmallmatrix} \ \ \
    E_{4}=\begin{psmallmatrix} 0 & 0 & 0 & 1\\
        0 & 0 & 0 & 0 \\ 0 & 0 & 0 & 0 \\ 1 & 0 & 0 & 0 \end{psmallmatrix} \ \ \ 
    E_{5}=\begin{psmallmatrix} 0 & 0 & 0 & 0 \\
        0 & 0 & 0 & 1\\ 0 & 0 & 0 & 0 \\ 0 & 1 & 0 & 0 \end{psmallmatrix} \ \ \ 
    E_{6}=\begin{psmallmatrix} 0 & 0 & 0 & 0 \\
        0 & 0 & 0 & 0 \\ 0 & 0 & 0 & 1 \\ 0 & 0 & -1 & 0 \end{psmallmatrix} 
\]
The map $\#:\mathfrak{so}_{0}(2,2)_{\mathrm{Ad}\rho} \rightarrow \mathfrak{so}_{0}(2,2)_{\mathrm{Ad}\rho^{*}}$ is defined by setting
\[
 (\#A)(B)=\iota(A,B) 
\]
thus
\[
    \#A=\sum_{i=1}^{6} \iota(A,E_{i})E_{i}^{*} \ ,
\]
where $E_{i}^{*}$ satisfies
\[
    E_{i}^{*}(E_{j})=\begin{cases} 1 \ \ \ \ \ \ \text{if $i=j$} \\
    0 \ \ \ \ \ \  \text{otherwise}
    \end{cases} \ .
\]
Applying Lemma \ref{lm:computation} to compute $\iota(E_{i}, E_{j})$, we obtain the following
\begin{align*}
    \#E_{1}&=\frac{4}{y^{2}}(E_{1}^{*}-x^2E_{3}^{*}+xE_{2}^{*}) \\
    \#E_{2}&=\frac{4}{y^{2}}(xE_{1}^{*}-x(x^2+y^2)E_{3}^{*}+(x^2+\tfrac{y^{2}}{2})E_{2}^{*})\\
    \#E_{3}&=\frac{4}{y^{2}}(-x^{2}E_{1}^{*}+(x^{2}+y^2)E_{3}^{*}-x(x^2+y^2)E_{2}^{*}) \ .
\end{align*}
Putting everything together, we get
\begin{align*}
  & \ \ \ \ d*(\#)\left(\phi(z)dz\otimes \Phi\begin{psmallmatrix} -z\ & z^2 \\ -1 & z \end{psmallmatrix}\right) \\
  &=d*[z^{2}\phi(z)dz\otimes\#E_{1}-\phi(z)dz\otimes\#E_{3}-2z\phi(z)dz\otimes\#E_{2}] \\
  &=d*[-4\phi(z)dz\otimes E_{1}^{*}+4z^{2}\phi(z)dz\otimes E_{3}^{*}-4z\phi(z)dz \otimes E_{2}^{*}] \ .
\end{align*}
Because $z$ is a conformal coordinate for the induced metric on the maximal surface, from the definition of the Hodge star operator we see that $*dx=dy$ and $*dy=-dx$, hence, extending the operator to complex $1$-forms by complex anti-linearity (i.e. $*(i\alpha)=-i*\bar{\alpha}$) we see that $*\phi(z)dz=i\overline{\phi(z)}d\bar{z}$. Therefore,
\begin{align*}
 & \ \ \ \ d*[-4\phi(z)dz\otimes E_{1}^{*}+4z^{2}\phi(z)dz\otimes E_{3}^{*}-4z\phi(z)dz \otimes E_{2}^{*}] \\
 &=d[-4i\overline{\phi(z)}d\bar{z}\otimes E_{1}^{*}+4i\overline{\phi(z)}\bar{z}^{2}d\bar{z}\otimes E_{3}^{*}-4i\overline{\phi(z)}\bar{z}d\bar{z}\otimes E_{2}^{*}]=0
\end{align*}
because $\phi(z)$ is holomorphic. As a consequence $\phi(z)dz\otimes \Phi\begin{psmallmatrix} -z\ & z^2 \\ -1 & z \end{psmallmatrix}$ is $d$-closed and $\delta$-closed, hence it is harmonic.
\end{proof}

We can finally prove one of the main results of the section:
\begin{teo} The metric $g$ on $\mathcal{GH}(S)$ restricts on the Fuchsian locus to a constant multiple of the Weil-Petersson metric on Teichm\"uller space.
\end{teo}
\begin{proof} By Lemma \ref{lm:tangent}, it is sufficient to show that
\[
 \tilde{g}\left(\phi(z)dz\otimes \Phi\begin{psmallmatrix} -z\ & z^2 \\ -1 & z \end{psmallmatrix}, \psi(z)dz\otimes \Phi\begin{psmallmatrix} -z\ & z^2 \\ -1 & z \end{psmallmatrix}\right)=\langle \phi, \psi\rangle_{WP}\ ,
\]
where here we are extending $\tilde{g}$ to an hermitian metric on the space of $\mathfrak{so}(4,\C)_{\mathrm{Ad}\rho}$-valued $1$-forms.
From the definition of $\tilde{g}$ and Corollary \ref{lm;computation_Fuchsian} we have
\begin{align*}
    & \ \ \ \ \tilde{g}\left(\phi(z)dz\otimes \Phi\begin{psmallmatrix} -z\ & z^2 \\ -1 & z \end{psmallmatrix}, \psi(z)dz\otimes \Phi\begin{psmallmatrix} -z\ & z^2 \\ -1 & z \end{psmallmatrix}\right)\\
    &=\mathcal{R}e\int_{S} \iota_{z}\left(\Phi\begin{psmallmatrix} -z\ & z^2 \\ -1 & z \end{psmallmatrix}, \Phi\begin{psmallmatrix} -z\ & z^2 \\ -1 & z \end{psmallmatrix}\right) \phi(z)dz\wedge *(\psi(z)dz) \\
    &=\mathcal{R}e\int_{S} \iota_{z}\left(\Phi\begin{psmallmatrix} -z\ & z^2 \\ -1 & z \end{psmallmatrix}, \Phi\begin{psmallmatrix} -z\ & z^2 \\ -1 & z \end{psmallmatrix}\right) \phi(z)dz\wedge (i\overline{\psi(z)}d\bar{z}) \\
    &=\mathcal{R}e\int_{S} 16i\phi(z)\overline{\psi(z)}y^2dz\wedge d\bar{z}\\
    &=32\langle \phi, \psi \rangle_{WP} \ .
\end{align*}
\end{proof}

Our next goal is to show that the Fuchsian locus is totally geodesic. To this aim it is sufficient to find an isometry of $(\mathcal{GH}(S), g)$ that fixes the Fuchsian locus. In Mess' parametrization of $\mathcal{GH}(S)$ as $\Teich(S) \times \Teich(S)$, there is a natural involution that swaps left and right representations, thus fixing pointwise the Fuchsian locus. Identifying $\SL(2,\R)\times \SL(2,\R)$ with $\mathrm{SO}_{0}(2,2)$, this corresponds to conjugation by $Q=\mathrm{diag}(-1,-1,1,-1) \in O(2,2)$. Therefore, we introduce the map
\begin{align*}
    q: \mathcal{GH}(S) &\rightarrow \mathcal{GH}(S) \\
        \rho &\mapsto Q\rho Q^{-1}
\end{align*}
and show that this is an isometry for the metric $g$. \\

\noindent We first need to compute the induced map in cohomology
\[
    q_{*}:H^{1}(S, \mathfrak{so}_{0}(2,2)_{\mathrm{Ad}\rho})\rightarrow H^{1}(S, \mathfrak{so}_{0}(2,2)_{\mathrm{Ad}q(\rho)})\ .
\]
It is well-known (see e.g. \cite{Goldman_symplectic}) that a tangent vector to a path of representations $\rho_{t}$ is a $1$-cocycle, that is a map $u:\pi_{1}(S) \rightarrow \mathfrak{so}_{0}(2,2)$ satisfying 
\[
    u(\gamma\gamma')-u(\gamma')=\mathrm{Ad}(\rho(\gamma))u(\gamma') \ .
\]
It is then clear that, if $u$ is a $1$ cocycle tangent to $\rho$, then $QuQ^{-1}$ is a $1$-cocycle tangent to $q(\rho)$. A $1$-cocycle represents a cohomology class in $H^{1}(\pi_{1}(S), \mathfrak{so}_{0}(2,2))$ which is isomorphic to $H^{1}(S, \mathfrak{so}_{0}(2,2)_{\mathrm{Ad}\rho})$ via 
\begin{align*}
 H^{1}(S, \mathfrak{so}_{0}(2,2)_{\mathrm{Ad}\rho}) &\rightarrow H^{1}(\pi_{1}(S), \mathfrak{so}_{0}(2,2)) \\
 [\sigma \otimes \phi] &\mapsto u_{\sigma \otimes \phi}: \gamma \mapsto \int_{\gamma}\sigma \otimes \phi \ .
\end{align*}

\begin{lemma}\label{lm:cohomology_map} For any $\sigma \in \mathcal{A}^{1}(S)$ and for any section $\phi$ of $\mathfrak{so}_{0}(2,2)_{\mathrm{Ad}\rho}$, we have
\[
    q_{*}[\sigma \otimes \phi]=[\sigma \otimes Q\phi Q^{-1}]
\]
\end{lemma}
\begin{proof}It is sufficient to show that $u_{\sigma \otimes Q\phi Q^{-1}}=Qu_{\sigma \otimes \phi}Q^{-1}$. This follows because, for any $\gamma \in \pi_{1}(S)$
\[
    \int_{\gamma} \sigma\otimes Q\phi Q^{-1}=Q\left(\int_{\gamma} \sigma \otimes \phi\right)Q^{-1} \ .
\]
\end{proof}

\noindent By an abuse of notation, we will still denote by $q_{*}$ the map induced by $q$ at the level of $\mathfrak{so}_{0}(2,2)_{\mathrm{Ad}\rho}$-valued $1$-forms. Our next step is to show that $q_{*}$ preserves the metric $\tilde{g}$.

\begin{lemma}\label{lm:preserve}For any $\sigma, \sigma' \in \mathcal{A}^{1}(S)$ and for any sections $\phi$ and $\phi'$ of $\mathfrak{so}_{0}(2,2)_{\mathrm{Ad}\rho}$, we have
\[
    \tilde{g}(q_{*}(\sigma\otimes \phi), q_{*}(\sigma'\otimes \phi'))=\tilde{g}(\sigma\otimes \phi, \sigma\otimes \phi) \ .
\]
\end{lemma}
\begin{proof} Given $\rho \in \mathcal{GH}(S)$, we denote by $M_{\rho}$ the GHMC anti-de Sitter manifold with holonomy $\rho$. Because $M_{\rho}$ and $M_{Q\rho Q^{-1}}$ are isometric via the map induced in the quotients by $Q: \widehat{AdS}_{3} \rightarrow \widehat{AdS}_{3}$, the minimal surfaces $\Sigma_{\rho}$ and $\Sigma_{Q\rho Q^{-1}}$ are isometric as well. In particular, their induced metric $h$ and $h^{q}$ coincide on every $\tilde{x} \in \tilde{S}$. Moreover, if $\tilde{\sigma}: \tilde{S} \rightarrow \widehat{AdS}_{3}$ is the $\rho$-equivariant maximal embedding, then $Q\tilde{\sigma}$ is $Q\rho Q^{-1}$-equivariant and still maximal. We deduce that if $H$ is a matrix representation of the $\rho$-equivariant inner product $\iota_{\tilde{x}}$ on $\tilde{S} \times \mathfrak{so}_{0}(2,2)$, then $H^{q}=Q^{t}HQ$ is the matrix representation of the $Q\rho Q^{-1}$-equivariant inner product $\iota_{\tilde{x}}^{q}$. Therefore, noting that $Q=Q^{t}=Q^{-1}$, for any $\phi$ and $\phi'$ sections of $\mathfrak{so}_{0}(2,2)_{\mathrm{Ad}\rho}$ and for any $p \in S$, we have
\begin{align*}
    \iota_{p}(\phi,\phi')&=\iota_{\tilde{x}}(A,B) \ \ \ \text{taking $\tilde{\phi}_{\tilde{x}}=A$, $\tilde{\phi}'_{\tilde{x}}=B$ and $\tilde{x} \in \pi^{-1}(p)$}\\
    &=\trace(A^{t}H^{-1}BH) \ \ \ \text{by Lemma \ref{lm:computation}}\\
    &=\trace(Q^{t}(Q^{t})^{-1}A^{t}Q^{t}(Q^{t})^{-1}QQ^{-1}H^{-1}(Q^{t})^{-1}Q^{t}Q^{-1}QBQ^{-1}Q(Q^{t})^{-1}Q^{t}HQQ^{-1}) \\
    &=\trace(Q^{t}(q_{*}(A))^{t}(H^{q})^{-1}q_{*}(B)H^{q}Q^{-1})\\
    &=\trace((q_{*}(A))^{t}(H^{q})^{-1}q_{*}(B)H^{q}Q^{-1}Q^{t})\\
    &=\trace((q_{*}(A))^{t}(H^{q})^{-1}q_{*}(B)H^{q})\\
    &=\iota_{\tilde{x}}^{q}(q_{*}(A), q_{*}(B))=\iota_{p}^{q}(q_{*}(\phi), q_{*}(\phi')) \ \ \ \text{by Lemma \ref{lm:cohomology_map}.}
\end{align*}
We can now compute
\begin{align*}
    \tilde{g}(\sigma \otimes \phi, \sigma'\otimes \phi')&= \int_{S} \iota(\phi, \phi')\sigma \wedge (*\sigma') \\
    &=\int_{S} \iota(\phi,\phi')\langle \sigma, \sigma' \rangle_{h} dA_{h} \\
    &=\int_{S} \iota^{q}(q_{*}(\phi), q_{*}(\phi')) \langle \sigma, \sigma' \rangle_{h^{q}} dA_{h^{q}} \\
    &=\tilde{g}(q_{*}(\sigma \otimes \phi, \sigma' \otimes \phi') \ \, 
\end{align*}
which shows that $q_{*}$ is an isometry for the Riemannian metrics on the bundles $\mathcal{A}^{1}(S, \mathfrak{so}_{0}(2,2)_{\mathrm{Ad}\rho})$ and $\mathcal{A}^{1}(S, \mathfrak{so}_{0}(2,2)_{\mathrm{Ad}q(\rho)})$
\end{proof}

\noindent In order to conclude that $q: \mathcal{GH}(S)\rightarrow \mathcal{GH}(S)$ is an isometry for $g$, it is sufficient now to show that the map $q_{*}$ preserves harmonicity of forms.

\begin{lemma} The map $q_{*}:H^{1}(S, \mathfrak{so}_{0}(2,2)_{\mathrm{Ad}\rho})\rightarrow H^{1}(S, \mathfrak{so}_{0}(2,2)_{\mathrm{Ad}q(\rho)})$ sends harmonic forms to harmonic forms.
\end{lemma}
\begin{proof}Let $\sum_{i}\sigma_{i} \otimes \phi_{i}$ be the harmonic representative in its cohomology class. This is equivalent to saying that $d(\sum_{i}\sigma_{i}\otimes \phi_{i})=0$ and $\delta(\sum_{i}\sigma_{i} \otimes \phi_{i})=0$. We need to show that these imply $d(\sum_{i}\sigma_{i}\otimes Q\phi_{i} Q^{-1})=0$ and $\delta(\sum_{i}\sigma_{i} \otimes Q\phi_{i} Q^{-1})=0$, as well. \\
The condition $d(\sum_{i}\sigma_{i} \otimes Q\phi_{i} Q^{-1})=0$ easily follows by linearity of $d$. \\
As for $\delta$-closedness, by definition of $\delta$, we have $\delta(\sum_{i}\sigma_{i}\otimes \phi_{i})=0$ if and only if $d*\#(\sum_{i}\sigma_{i} \otimes \phi_{i})=d*(\sum_{i}\sigma_{i} \otimes \#\phi_{i})=0$. 
Let us denote by $\#^{q}$ the analogous operator defined on $\mathfrak{so}_{0}(2,2)_{\mathrm{Ad}q(\rho)}$-valued $1$-forms. Let $\{E_{j}\}_{j=1}^{6}$ be the basis of $\mathfrak{so}_{0}(2,2)$ introduced in the proof of Lemma \ref{lm:tangent} and denote by $\{E_{j}^{*}\}_{j=1}^{6}$ its dual. By definition of $\#$ and $\#^{q}$ we have
\[
    \#A=\sum_{j=1}^{6}\iota(A,E_{j})E_{j}^{*} \ \ \ \text{and} \ \ \
    \#^{q}A=\sum_{j=1}^{6}\iota^{q}(A,E_{j})E_{j}^{*} \ ,
\]
where, as in Lemma \ref{lm:preserve}, we denoted by $\iota^{q}$ the inner product on $\mathfrak{so}_{0}(2,2)_{\mathrm{Ad}q(\rho)}$. Hence, $d*(\sum_{i}\sigma_{i} \otimes \#\phi_{i})=0$ if and only if 
\[
    d*\left(\sum_{i}\sigma_{i}\otimes \sum_{j=1}^{6}\iota(\phi_{i}, E_{j})E_{j}^{*}\right)=0  \ ,                                                                                                                                                                                                                                                                                                          \]
which implies that 
\begin{equation}\label{eq:relation}
d*\left(\sum_{i}\sigma_{i}\iota(\phi_{i}, E_{j})\right)=0 \ \ \ 
\text{for every $j=1,\dots,6$ .} 
\end{equation}
Therefore, using that $\iota^{q}(QAQ^{-1}, QBQ^{-1})=\iota(A,B)$ for every $A,B \in \mathfrak{so}_{0}(2,2)$, we have
\begin{align*}
    d*\left(\sum_{i}\sigma_{i}\otimes \#^{q}Q\phi_{i} Q^{-1}\right)&=d*\left(\sum_{i}\sigma_{i}\otimes \sum_{j=1}^{6}\iota^{q}(Q\phi_{i} Q^{-1},E_{j})E_{j}^{*}\right)\\
    &=d*\left(\sum_{i}\sigma_{i} \otimes \sum_{j=1}^{6}\iota^{q}(Q\phi_{i} Q^{-1}, QQ^{-1}E_{j}QQ^{-1})E_{j}^{*}\right)\\
    &=d*\left(\sum_{i}\sigma_{i} \otimes  \sum_{j=1}^{6}\iota(\phi_{i}, Q^{-1}E_{j}Q)E_{j}^{*}\right) \ .
\end{align*}
A straightforward computation shows that
\begin{align*}
    &Q^{-1}E_{1}Q=-E_{3} \ \ &Q^{-1}E_{2}Q=-E_{2} \ \ \ \ \ \ \ \
    &Q^{-1}E_{3}Q=-E_{1} \\
    &Q^{-1}E_{4}Q=E_{4} \ \ \
    &Q^{-1}E_{5}Q=E_{5} \ \ \ \ \ \ \ \ \ \ \
    &Q^{-1}E_{6}Q=-E_{6} \ \ \
\end{align*}
thus $d*(\sum_{i}\sigma_{i} \otimes  \sum_{j=1}^{6}\iota(\phi_{i}, Q^{-1}E_{j}Q)E_{j}^{*})=0$, because, up to a sign, the coefficients of $E_{j}^{*}$ coincide with those in Equation \ref{eq:relation} for $j\neq 1,3$ and the coefficient of $E_{1}^{*}$ is swapped with that of $E_{3}^{*}$ in Equation \ref{eq:relation}. Hence, $d*\#(\sum_{i}\sigma_{i}\otimes Q\phi_{i} Q^{-1})=0$, and then $\delta(\sum_{i}\sigma_{i} \otimes Q\phi_{i} Q^{-1})=0$, as required.
\end{proof}

\noindent Combining the above result with Lemma \ref{lm:preserve}, by definition of the metric $g$ on $\mathcal{GH}(S)$ we obtain the following:
\begin{teo}The map $q:\mathcal{GH}(S)\rightarrow \mathcal{GH}(S)$ is an isometry for $g$. In particular, the Fuchsian locus, which is pointwise fixed by $q$, is totally geodesic.
\end{teo}

\bibliographystyle{alpha}
\bibliography{Reference}

\end{document}